\documentclass[11pt, amsfonts]{article}

\usepackage{amsmath,amssymb,amsthm}
\usepackage[all]{xy}
\usepackage{color}

\SelectTips{eu}{12}

\newtheorem{thm}{Theorem}[section]
\newtheorem{prop}[thm]{Proposition}
\newtheorem{lem}[thm]{Lemma}
\newtheorem{cor}[thm]{Corollary}

\theoremstyle{definition}

\theoremstyle{remark}
\newtheorem{rem}[thm]{Remark}

\newcommand{\pt}{{\rm pt}}

\newcommand{\conn}{{\rm conn}}

\title{Matching complexes of polygonal line tilings}

\author{Takahiro Matsushita\thanks{Supported by JSPS KAKENHI 19K14536.}\\
Department of Mathematical Sciences, The University of The Ryukyus,\\ Nishihara-cho, Okinawa, Japan\\
\tt mtst@sci.u-ryukyu.ac.jp}

\begin{document}

\baselineskip.525cm

\maketitle

\begin{abstract}
The matching complex of a simple graph $G$ is the simplicial complex consisting of the matchings on $G$. Jeli\'c Milutinovi\'c et al.~\cite{JMMV} studied the matching complexes of the polygonal line tilings, and they gave a lower bound for the connectivity of the matching complexes of polygonal line tilings. In this paper, we determine the homotopy types of the matching complexes of polygonal line tilings recursively, and determine their connectivities.

\bigskip
\noindent {\bf Mathematical Subject of Classification:} Primary:55P10, Secondary:05C69, 05C70.
\end{abstract}

\section{Introduction}

A matching on a simple graph $G$ is a set of edges in $G$ which have no common vertex. The matching complex $M(G)$ is the simplicial complex whose vertex set is the edge set $E(G)$ and whose simplices are the matchings on $G$.

Bouc \cite{Bouc} introduced the matching complexes of complete graphs in the context of the Brown complexes and Quillen complexes of finite groups. The matching complexes of complete bipartite graphs are called chessboard complexes. There are many works on these two complexes, and we refer to \cite{Wachs} for an introduction to this subject. The integral homology groups of these complexes have many torsions (see \cite{SW}) in general, and hence it is very difficult to determine the homotopy types of matching complexes.

On the other hand, studies on homotopy types of matching complexes of other graphs have sporadically appeared. First Kozlov \cite{Kozlov} determined the homotopy types of matching complexes of path graphs and cycle graphs. Marietti and Testa showed that the matching complexes of forests are homotopy equivalent to wedges of spheres. Jonsson studied the topology of the matching complexes of grid graphs in his unpublished work \cite{Jonsson}, and the author \cite{Matsushita} determined the homotopy types of the matching complexes of $(t \times 2)$-grid graphs, and showed that the matching complex of $(t \times 2)$-grid graphs are homotopy equivalent to wedges of spheres (see also Braun and Hough \cite{BH}).

Recently, Jeli\'c Milutinovi\'c et al.~\cite{JMMV} studied the matching complexes of polygonal line tilings. Polygonal line tilings are a {generalization} of $(t \times 2)$-grid graphs and some of honeycomb graphs defined as follows: For an integer $n$ greater than $2$ and a non-negative integer $t$, the graph $P_{n,t}$ of $t$ $(2n)$-gons is the graph in which $t$ $(2n)$-gons are arranged in a straight line as shown in Figure 1. Namely, the graph of $t$ $4$-gons is the $((t+1) \times 2)$-grid graph, and the graph of $t$ $6$-gons is {called} the $(1 \times 1 \times t)$-honeycomb graph in \cite{JMMV}. Using discrete Morse theory, they showed that the matching complexes of polygonal line tilings are homotopy equivalent to much smaller complexes, and gave lower bounds for the connectivities of these complexes. In particular, they determined the homotopy types of $M(P_{3m+1,t})$ for every positive integer $m.$

\begin{figure}[t]
\begin{center}
\begin{picture}(220,130)(0,0)
\multiput(0,90)(40,0){6}{\circle*{3}}
\multiput(0,110)(40,0){6}{\circle*{3}}
\multiput(20,80)(40,0){5}{\circle*{3}}
\multiput(20,120)(40,0){5}{\circle*{3}}
\multiput(0,90)(40,0){5}{\line(2,-1){20}}
\multiput(0,90)(40,0){6}{\line(0,1){20}}
\multiput(0,110)(40,0){5}{\line(2,1){20}}
\multiput(40,90)(40,0){5}{\line(-2,-1){20}}
\multiput(40,110)(40,0){5}{\line(-2,1){20}}

\put(220,95){$t = 5, n = 3$}

\multiput(10,25)(45,0){5}{\circle*{3}}
\multiput(10,40)(45,0){5}{\circle*{3}}
\multiput(25,10)(45,0){4}{\circle*{3}}
\multiput(40,10)(45,0){4}{\circle*{3}}
\multiput(25,55)(45,0){4}{\circle*{3}}
\multiput(40,55)(45,0){4}{\circle*{3}}
\multiput(10,25)(45,0){4}{\line(1,-1){15}}
\multiput(10,25)(45,0){5}{\line(0,1){15}}
\multiput(10,40)(45,0){4}{\line(1,1){15}}
\multiput(25,10)(45,0){4}{\line(1,0){15}}
\multiput(25,55)(45,0){4}{\line(1,0){15}}
\multiput(55,25)(45,0){4}{\line(-1,-1){15}}
\multiput(55,40)(45,0){4}{\line(-1,1){15}}

\put(210,28){$t = 4, n = 4$}
\end{picture}

{\bf Figure 1.}
\end{center}
\end{figure}

The purpose in this paper is to determine the homotopy types of matching complexes of polygonal line tilings. It turns out that these complexes are homotopy equivalent to wedges of spheres, and we determine the precise connectivities of these complexes. To explain our main results precisely, we need some preparation.

The independence complex of a graph $G$ is the simplicial complex whose simplices are the independent sets of $G$. The line graph of $G$ is the graph $L(G)$ whose vertex set is the edge set $E(G)$ of $G$ and two elements of $E(G)$ are adjacent if they have a common vertex. Then the matching complex $M(G)$ is the independence complex $I(L(G))$ of the line graph.

Let $G_{n,t}$ be the line graph of the graph $P_{n,t}$ of $t$ $(2n)$-gons. Thus we have $I(G_{n,t}) = M(P_{n,t})$. We can formulate the graph $G_{n,t}$ as follows: The vertex set of $G_{n,t}$ is
$$V(G_{n,t}) = \{ a_0, \cdots, a_t\} \cup \{ b_{i,j}, c_{i,j} \; | \; 1 \le i \le t, \; 1 \le j < n\}.$$
The adjacent relation of $G_{n,t}$ is formulated as follows:
\begin{itemize}
\item $a_i \sim b_{i+1,1}$, {$a_i \sim c_{i+1,1}$}, {$a_{i+1} \sim b_{i+1,n-1}$}, and $a_{i+1} \sim c_{i+1, n-1}$ for $i = 0, \cdots, t-1$.

\item $b_{i,j} \sim b_{i, j + 1}$ and $c_{i,j} \sim c_{i, j + 1}$ for $i = 1, \cdots, t$ and $j = 1, \cdots, n-2$.

\item $b_{i,n-1} \sim b_{i+1, 1}$ and $c_{i, n -1} \sim c_{i+1,1}$ for {$i = 1, \cdots, t-1$}.
\end{itemize}
Figure 2 depicts the graph $G_{n,t}$ when $n = 4$ and $t = 3$. Here we consider that $G_{n,0}$ is the graph with one vertex $a_0$. Thus $I(G_{n,0})$ is a point. Using this notation, our main theorems are formulated as follows:

\begin{thm}[Theorem \ref{thm 4.1.1} and Proposition \ref{prop 4.1.2}] \label{main thm 1}
Suppose that $n = 3m + 2$ for some non-negative integer $m$. Then the following hold:
\begin{itemize}
\item[(1)] For $t \ge 4$, there is a homotopy equivalence
$$I(G_{n,t}) \simeq \Sigma^{6m+2} I(G_{n,t-3}) \vee \Sigma^{6m+2} I(G_{n, t - 3}) \vee \Sigma^{8m+3} I(G_{n, t-4}).$$

\item[(2)] There are homotopy equivalences
$$I(G_{n,1}) \simeq S^{2m}, \; I(G_{n,2}) \simeq S^{4m+1} \vee S^{4m+1}, \; I(G_{n,3}) \simeq S^{6m+2}.$$
\end{itemize}
\end{thm}

Note that $I(G_{2,t})$ is the matching complex of the $((t+1) \times 2)$-grid graph. In fact, the case $m = 0$ in Theorem \ref{main thm 1} was proved by the author \cite{Matsushita}.

The corresponding statement in the case $n = 3m$ is a little complicated. Let $H_{n,t}$ be the induced subgraph of $G_{n,t+1}$ whose vertex set is $V(G_{n,t}) \cup \{ b_{t+1,1}, c_{t+1,1}\}$ (see Figure 2). Then the following holds:

\begin{figure}[t]
\begin{center}
\begin{picture}(360,170)(0,0)

\multiput(20,100)(40,0){8}{\circle*{3}}
\multiput(20,160)(40,0){8}{\circle*{3}}
\multiput(0,130)(80,0){5}{\circle*{3}}

\multiput(0,130)(80,0){4}{\line(2,-3){20}}
\multiput(0,130)(80,0){4}{\line(2,3){20}}
\multiput(80,130)(80,0){4}{\line(-2,3){20}}
\multiput(80,130)(80,0){4}{\line(-2,-3){20}}

\multiput(20,100)(0,60){2}{\line(10,0){280}}

\put(350,127){$G_{3,4}$}

\put(-14,128){\small $a_0$}
\put(64,128){\small $a_1$}
\put(144,128){\small $a_2$}
\put(224,128){\small $a_3$}
\put(304,128){\small $a_4$}

\put(15,167){\small $b_{1,1}$}
\put(55,167){\small $b_{1,2}$}
\put(95,167){\small $b_{2,1}$}
\put(135,167){\small $b_{2,2}$}
\put(175,167){\small $b_{3,1}$}
\put(215,167){\small $b_{3,2}$}
\put(255,167){\small $b_{4,1}$}
\put(295,167){\small $b_{4,2}$}

\put(15,90){\small $c_{1,1}$}
\put(55,90){\small $c_{1,2}$}
\put(95,90){\small $c_{2,1}$}
\put(135,90){\small $c_{2,2}$}
\put(175,90){\small $c_{3,1}$}
\put(215,90){\small $c_{3,2}$}
\put(255,90){\small $c_{4,1}$}
\put(295,90){\small $c_{4,2}$}

\put(335,77){\small $b_{5,1}$}
\put(335,0){\small $c_{5,1}$}

\multiput(20,10)(40,0){9}{\circle*{3}}
\multiput(20,70)(40,0){9}{\circle*{3}}
\multiput(0,40)(80,0){5}{\circle*{3}}

\multiput(0,40)(80,0){5}{\line(2,-3){20}}
\multiput(0,40)(80,0){5}{\line(2,3){20}}
\multiput(80,40)(80,0){4}{\line(-2,3){20}}
\multiput(80,40)(80,0){4}{\line(-2,-3){20}}

\multiput(20,10)(0,60){2}{\line(10,0){320}}

\put(350,37){$H_{3,4}$}

\end{picture}

{\bf Figure 2.}
\end{center}
\end{figure}

\begin{thm}[Theorem \ref{thm G_t}, Theorem \ref{thm H_t}, and Proposition \ref{prop 3.2.5}] \label{main thm 2}
Suppose that $n = 3m$ for some positive integer $m$. Then the following hold:
\begin{itemize}
\item[(1)] If $t \ge 4$, then there is a homotopy equivalence
$$I(G_{n,t}) \simeq \Sigma^{4m-2} I(H_{n,t-2}) \vee \Sigma^{6m-2} I(G_{n,t-3}) \vee \Sigma^{8m-3} I(H_{n,t-4}).$$

\item[(2)] If $t \ge 1$, then there is a homotopy equivalence
$$I(H_{n,t}) \simeq \Sigma^{2m}I(G_{n,t-1}) \vee \Sigma^{2m-1} I(H_{n,t-1}).$$

\item[(3)] There are homotopy equivalences
$$I(G_{n,1}) \simeq S^{2m-1} \vee S^{2m-1}, \; I(G_{n,2}) \simeq S^{4m-2} \vee S^{4m-2}, \; I(G_{n,3}) \simeq S^{6m-2} \vee S^{6m-3}, \; I(H_0) \simeq S^0.$$
\end{itemize}
\end{thm}

Here we mention the connectivity of the case $n = 3$. In this case, the polygonal line tiling is $(1 \times 1 \times t)$-honeycomb graph. In \cite{Jonsson book} Jonsson considered that the study of the matching complexes of honeycomb graphs seems to be interesting, and Jecli\'c Milutinovi\'c et al.~first studied them and showed that the matching complex of $(1 \times 1 \times t)$-honeycomb graph is $(t- 1)$-connected. In fact, Theorem \ref{main thm 2} implies that their estimate is strict:

\begin{cor}[Corollary \ref{cor conn honeycomb}]
The connectivity of the matching complex of $(1 \times 1 \times t)$-honeycomb graph is $t - 1$.
\end{cor}

Theorem \ref{main thm 1} and Theorem \ref{main thm 2} imply that the complex $I(G_{n,t})$ is a wedge of spheres when $n = 3m$ or $n = 3m+2$. Since Jeli\'c Milutinovi\'c et al.~\cite{JMMV} showed that $I(G_{n,t})$ is homotopy equivalent to a wedge of spheres when $n = 3m+1$, we conclude the following:

\begin{cor}
The matching complexes of polygonal line tilings are wedges of spheres.
\end{cor}

The rest in this paper is organized as follows: In Section 3, we review necessary facts concerning independence complexes. In Section 3, we determine the homotopy types of the matching complexes of polygonal line tilings. Since the proofs are different in the cases $n = 3m, 3m+1, 3m+2$, we divide this section into three parts. In Subsection 3.1, we determine the homotopy type of $I(G_{n,t})$ in the case $n = 3m+1$. This is an alternative proof of a result in Jeli\'c Milutinovi\'c et al.~\cite{JMMV}. We prove Theorem \ref{main thm 1} and Theorem \ref{main thm 2} in Subsection 3.2 and Subsection 3.3, respectively. In these sections, we determine the connectivity of $I(G_{n,t})$ (Theorem \ref{thm connectivity 1} and Theorem \ref{thm connectivity 2}).

\vspace{2mm}
\noindent {\bf Acknowledgement.} The author thanks Jeli\'c Milutinovi\'c for answering his question concerning her paper \cite{JMMV}. {The author also thanks the referee for useful comments.} The author is supported by JSPS KAKENHI 19K14536.

\section{Independence complexes}

In this section, we review some known results of independence complexes. For basic terminology concerning simplicial complexes in topological combinatorics, we refer to \cite{Jonsson book} or \cite{Kozlov book}.

Throughout the paper, graphs are assumed to be finite and simple. For a subset $S$ of the vertex set $V(G)$ of $G$, let $G[S]$ denote the subgraph in $G$ induced by $S$. We write $G \setminus S$ to mean $G [V(G) \setminus S]$. For a vertex $v$ in $G$, we write $G \setminus v$ instead of $G \setminus \{ v\}$.

Recall that a subset $S$ of $V(G)$ is independent if $G[S]$ has no edges. The {\it independence complex $I(G)$ of $G$} is the (abstract) simplicial complex whose vertex set is $V(G)$ and whose simplices are the independent sets of $G$.

The following proposition is easily deduced from the definition.

\begin{prop} \label{prop join}
If $G$ is a disjoint union $G_1 \sqcup G_2$ of subgraphs $G_1$ and $G_2$, then $I(G) = I(G_1) * I(G_2)$.
\end{prop}

Here $*$ denotes the join of simplicial complexes. Thus if $G$ has a connected component whose independence {complex} is contractible, then $I(G)$ is also contractible. In particular, if $G$ has an isolated vertex, then $I(G)$ is contractible.

For a simplicial complex $K$ and a vertex $v$ in $K$, we write $K \setminus v$ to indicate the simplicial complex whose simplices are the simplices in $K$ not containing $v$. Let ${\rm link}_K(v)$ be the link of $K$ at $v$. Clearly, there is a cofiber sequence
$${\rm link}_K(v) \to K \setminus v \to K.$$

Now we apply this cofiber sequence to independence complexes. For a vertex $v$ in $G$, the {\it open neighborhood of $v$} is $N_G(v) = \{ w \in V(G) \; | \; \{ v,w\} \in E(G)\}$, and the {\it closed neighborhood $N_G[v]$} is $N_G(v) \cup \{ v\}$. If we do not need to mention $G$, we often write $N(v)$ (or $N[v]$) instead of $N_G(v)$ (or $N_G[v]$, respectively). It is easy to see that ${\rm link}_{I(G)}(v)$ is $I(G \setminus N_G[v])$ and $I(G) \setminus v = I(G \setminus v)$. Thus we have the following:

\begin{prop}[See Adamaszek \cite{Adamaszek}] \label{prop cofib}
The sequence
$$I(G \setminus N[v]) \hookrightarrow I(G \setminus v) \hookrightarrow I(G).$$
is a cofiber sequence. In particular, if the inclusion $I(G \setminus N[v]) \hookrightarrow I(G \setminus v)$ is null-homotopic, then there is a homotopy equivalence $I(G) \simeq I(G \setminus v) \vee \Sigma I(G \setminus N[v])$.
\end{prop}

We need the following two propositions from \cite{Engstrom}.

\begin{prop}[Fold lemma, Lemma 3.2 of \cite{Engstrom}]
Let $v$ and $w$ be distinct vertices in $G$, and suppose $N(w) \subset N(v)$. Then the inclusion $I(G \setminus v) \hookrightarrow I(G)$ is a homotopy equivalence.
\end{prop}

A vertex $v$ in $G$ is {\it simplicial} if $G [N(v)]$ is a complete graph.

\begin{thm}[Theorem 3.7 of \cite{Engstrom}] \label{thm simplicial}
If $v$ is a simplicial vertex in $G$, there is a homotopy equivalence
$$I(G) \simeq \bigvee_{w \in N(v)} \Sigma I(G \setminus N[w]).$$
\end{thm}

Let $P_k$ be the path graph with $k$ vertices. Namely, $V(P_k) = \{ 1, \cdots, k\}$ and $E(P_k) = \{ \{ i,j\} \; | \; |i - j| = 1\}$. Let $C_k$ be the $k$-cycle graph. We will frequently use the following basic computation by Kozlov \cite{Kozlov}.

\begin{prop}[Proposition 4.6 and 5.2 of Kozlov \cite{Kozlov}]
The following hold:
\begin{itemize}
\item[(1)] $I(P_{3k}) \simeq S^{k-1}$, $I(P_{3k+1}) \simeq \pt$, $I(P_{3k+2}) \simeq S^k$ for $k \ge 0$.

\item[(2)] $I(C_{3k}) \simeq S^{k-1} \vee S^{k-1}$, $I(C_{3k+1}) \simeq S^{k-1}$, $I(C_{3k+2}) \simeq S^k$ for $k \ge 1$.
\end{itemize}
\end{prop}

\begin{cor} \label{cor P}
For $k \ge 0$, the inclusion $I(P_{3k+2}) \hookrightarrow I(P_{3k+3})$ is a homotopy equivalence.
\end{cor}
\begin{proof}
Put $v = 3k+3 \in V(P_{3k+3})$. Then $I(P_{3k+3} \setminus N_{P_{3k+3}}[v]) = I(P_{3k+1})$ is contractible. Thus Proposition \ref{prop cofib} implies that  the inclusion $I(P_{3k+2}) \hookrightarrow I(P_{3k+3})$ is a homotopy equivalence.
\end{proof}

Finally, we give a rather technical argument which makes the description much simpler in Section 3. First, we introduce the following terminology: An {\it $n$-string} based at $v$ of a graph $G$ is an induced subgraph $P$ of $G$ which satisfies the following:
\begin{itemize}
\item[(1)] $P$ is isomorphic to $P_n$, and $v$ is an endpoint of $P$.

\item[(2)] Every vertex in $P$ except for $v$ is adjacent to no vertex belonging to $V(G) \setminus V(P)$.
\end{itemize}

\begin{prop} \label{prop string}
Let $G$ be a graph and $P$ a $(3k+2)$-string based at $v$ in $G$. Let $w$ be the vertex in $P$ which is adjacent to $v$, and set $U = N_G(v) \setminus w$. Then the inclusion $I(G \setminus U) \hookrightarrow I(G)$ is a homotopy equivalence.
\end{prop}
\begin{proof}
Set $U = \{ u_1, \cdots, u_n\}$. Then $P \setminus v$ is a connected component of the graph $G \setminus N_G[u_1]$, which is isomorphic to $P_{3k+1}$. Since $I(P_{3k+1})$ is contractible, Proposition \ref{prop join} implies that $I(G \setminus N_G[u_1])$ is contractible. Thus it follows from Proposition \ref{prop cofib} that the inclusion $I(G \setminus u_1) \hookrightarrow I(G)$ is a homotopy equivalence. By the same reason, the inclusion $I(G \setminus \{ u_1, u_2\}) \hookrightarrow I(G \setminus u_1)$ is a homotopy equivalence. Iterating this, we have that the inclusion $I(G \setminus U) = I(G \setminus \{ u_1, \cdots u_n\}) \hookrightarrow I(G)$ is a homotopy equivalence.
\end{proof}

\section{Computations}

In this section, we determine the homotopy types and connectivities of the matching complexes of the graph of $t$ $(2n)$-gons. Recall that $G_{n,t}$ is the line graph of the {graphs} $P_{n,t}$ of $t$ $(2n)$-gons. {An explicit formulation of $G_{n,t}$ was given in Section 1.}




From now on, we fix an integer $n$ greater than $1$ and write $G_t$ instead of $G_{n,t}$. The proof is devided into three cases.

\subsection{Case $n = 3m+1$}

In this subsection, we determine the homotopy types of $I(G_t)$ when $n = 3m+1$. As was mentioned in Section 1, Jeli\'c Milutinovi\'c et al.~\cite{JMMV} determined the homotopy types of $I(G_t)$ in this case. Here we give an alternative proof of their result.

\begin{thm}[Corollary 4.3 of \cite{JMMV}] \label{thm 3m+1}
$I(G_t) \simeq \bigvee_t S^{2tm}$
\end{thm}

Throughout this subsection, put $X_t = G_t \setminus N[a_t]$ for $t \ge 1$ (In the next two subsections, we use $X_t$ to indicate other graphs).

\begin{lem} \label{lem 3m+1}
$I(X_t) \simeq S^{2tm-1}$
\end{lem}
\begin{proof}
We prove this by induction on $t$. Since $I(X_1) = I(P_{6m-1}) \simeq S^{2m-1}$, the case $t = 1$ holds. Suppose $t \ge 2$. Then $b_{t, 3m-1}, \cdots, b_{t,1}$ form an $(3m-1)$-string based at $b_{t,1}$, and $c_{t, 3m-1}, \cdots, c_{t,1}$ form an $(3m-1)$-string based at $c_{t,1}$. Using Proposition \ref{prop string}, we can delete the three vertices $b_{t-1, 3m-1}, c_{t-1,3m-1}, a_{t-1}$ from $X_t$. Figure 3 depicts the case $m = 1$. Hence we have
$$I(X_t) \simeq I(P_{3m-1}) * I(P_{3m-1}) * I(X_{t-1}) \simeq S^{m-1} * S^{m-1} * I(X_{t-1}) \simeq \Sigma^{2m} I(X_{t-1}).$$
This completes the proof.
\end{proof}

\begin{figure}[t]
\begin{center}
\begin{picture}(420,120)(0,0)

\multiput(20,10)(20,0){7}{\circle*{3}}
\multiput(30,30)(60,0){2}{\circle*{3}}
\multiput(20,50)(20,0){7}{\circle*{3}}

\multiput(20,10)(60,0){2}{\line(1,2){20}}
\multiput(20,50)(60,0){2}{\line(1,-2){20}}

\put(150,30){\circle{3}}

\multiput(20,10)(0,40){2}{\line(1,0){120}}
\multiput(20,10)(0,40){2}{\line(-1,0){10}}

\put(-25,28){\small $G_t \setminus a_t$}
\put(350,28){\small $G_{t-1} \sqcup P_{3m-1} \sqcup P_{3m-1}$}
\put(350,88){\small $X_{t-1} \sqcup P_{3m-1} \sqcup P_{3m-1}$}

\put(-15,88){\small $X_t$}

\put(173,28){\small $\simeq$}


\multiput(210,10)(20,0){4}{\circle*{3}}
\multiput(220,30)(60,0){2}{\circle*{3}}
\multiput(210,50)(20,0){4}{\circle*{3}}

\multiput(290,10)(0,40){2}{\circle{3}}
\multiput(310,10)(0,40){2}{\circle*{3}}
\multiput(330,10)(0,40){2}{\circle*{3}}
\multiput(310,10)(0,40){2}{\line(1,0){20}}

\put(280,30){\line(-1,2){10}}
\put(280,30){\line(-1,-2){10}}

\put(210,10){\line(1,2){20}}
\put(210,50){\line(1,-2){20}}

\put(340,30){\circle{3}}

\multiput(210,10)(0,40){2}{\line(1,0){60}}
\multiput(210,10)(0,40){2}{\line(-1,0){10}}


\multiput(20,70)(20,0){6}{\circle*{3}}
\multiput(30,90)(60,0){2}{\circle*{3}}
\multiput(20,110)(20,0){6}{\circle*{3}}

\multiput(20,70)(60,0){2}{\line(1,2){20}}
\multiput(20,110)(60,0){2}{\line(1,-2){20}}

\put(140,70){\circle{3}}
\put(140,110){\circle{3}}
\put(150,90){\circle{3}}

\multiput(20,70)(0,40){2}{\line(1,0){100}}
\multiput(20,70)(0,40){2}{\line(-1,0){10}}

\put(173,88){\small $\simeq$}


\multiput(210,70)(20,0){3}{\circle*{3}}
\put(220,90){\circle*{3}}
\multiput(280,90)(60,0){2}{\circle{3}}
\multiput(210,110)(20,0){3}{\circle*{3}}

\multiput(270,70)(0,40){2}{\circle{3}}
\multiput(290,70)(0,40){2}{\circle*{3}}
\multiput(310,70)(0,40){2}{\circle*{3}}

\put(210,70){\line(1,2){20}}
\put(210,110){\line(1,-2){20}}

\put(330,70){\circle{3}}
\put(330,110){\circle{3}}
\put(340,90){\circle{3}}

\multiput(210,70)(0,40){2}{\line(1,0){40}}
\multiput(210,70)(0,40){2}{\line(-1,0){10}}
\multiput(290,70)(0,40){2}{\line(1,0){20}}

\end{picture}

{\bf Figure 3.}
\end{center}
\end{figure}

\noindent
{\it Proof of Theorem \ref{thm 3m+1}.} Since $I(G_1) = I(C_{6m+2}) \simeq S^{2m}$, the case $t = 1$ holds. In $G_t \setminus a_t$, the vertices $b_{t, 3m}, \cdots, b_{t,2}$ form a $(3m-1)$-string based at $b_{t,2}$, and the vertices $c_{t, 3m}, \cdots, c_{t,2}$ form a $(3m - 1)$-string based at $c_{t,2}$. Thus using Proposition \ref{prop string}, we can delete the vertices $b_{t,1}$ and $c_{t,1}$ from $G_t \setminus a_t$ (see Figure 3). Hence we have
$$I(G_t \setminus a_t) \simeq I(P_{3m-1}) * I(P_{3m-1}) * I(G_{t-1}) \simeq \Sigma^{2m} I(G_{t-1}) \simeq \bigvee_{t-1} S^{2tm}.$$
Since $I(G_t \setminus N[a_t]) = I(X_t) \simeq S^{2tm-1}$, the inclusion $I(G_t \setminus N[a_t]) \hookrightarrow I(G_t \setminus a_t)$ is null-homotopic. Thus Proposition \ref{prop cofib} implies
$$I(G_t) \simeq I(G_t \setminus a_t) \vee \Sigma I(X_t) \simeq \bigvee_{t} S^{2tm}.$$
\qed

\subsection{Case $n = 3m+2$}

The purpose in this section is to determine the homotopy type and connectivity of $I(G_t)$ when $n = 3m + 2$. Here we assume that $m \ge 0$. Note that the case $m = 0$, the graph of $t$ $(2n)$-gons is a $(t \times 2)$-grid graph, and hence the homotopy types of their matching complexes are determined in \cite{Matsushita}. We first prove the following recursive formula:

\begin{thm} \label{thm 4.1.1}
For $t \ge 3$, there is a homotopy equivalence
$$I(G_t) \simeq \Sigma^{6m+2} I(G_{t-3}) \vee \Sigma^{6m+2} I(G_{t-3}) \vee \Sigma^{8m+3} I(G_{t-4}).$$
\end{thm}

For $t \ge 1$, we put $X_t = G_t \setminus a_{t-1}$ (see Figure 4).

\begin{lem} \label{lem 3.2.1}
For $t \ge 1$, there is a homotopy equivalence
$$I(G_t) \simeq I(X_t).$$
\end{lem}
\begin{proof}
By Proposition \ref{prop cofib}, it suffices to show that $I(G_t \setminus N[a_{t-1}])$ is contractible. {Note that} $G_t \setminus N[a_{t-1}]$ has a connected component consisting of the vertices
$$b_{t,2}, \cdots, b_{t, 3m+1}, a_t, c_{t,3m+1}, \cdots, c_{t,2},$$
which is isomorphic to $P_{6m+1}$. Since $I(P_{6m+1})$ is contractible, Proposition \ref{prop join} implies that $I(G_t \setminus N[a_{t-1}])$ is contractible.
\end{proof}

Next we put $Y_t =  X_t \setminus a_{t-2}$ for $t \ge 2$ (see Figure 4).

\begin{lem} \label{lem 3.2.2}
For $t \ge 3$, there is a homotopy equivalence
$$I(X_t) \simeq I(Y_t) \vee \Sigma^{6m+2} I(G_{t-3}).$$
\end{lem}
\begin{proof}
By Proposition \ref{prop cofib}, it suffices to show the following assertions:
\begin{itemize}
\item[(1)] The inclusion $I(X_t \setminus N[a_{t-2}]) \hookrightarrow I(X_t \setminus a_{t-2}) = I(Y_t)$ is null-homotopic.

\item[(2)] $I(X_t \setminus N[a_{t-2}]) \simeq \Sigma^{6m+1} I(G_{t-3})$.
\end{itemize}

We first show (1). The graph $X_t \setminus N[a_{t-2}]$ has a connected component $P$ consisting of
$$b_{t-1,2}, \cdots, b_{t-1, 3m+1}, b_{t,1}, \cdots, b_{t,3m+1}, a_t, c_{t,3m+1}, \cdots, c_{t,1}, c_{t-1, 3m+1}, \cdots, c_{t-1,2},$$
which is isomorphic to $P_{12m+3}$. We define the subgraph $H$ by $X_t \setminus N[a_{t-2}] = P \sqcup H$. 

Since the inclusion $I(P_{12m + 2}) \hookrightarrow I(P_{12m+3})$ is a homotopy equivalence (Corollary \ref{cor P}), we have that $I(P \setminus c_{t-1, 2}) \hookrightarrow I(P)$ is a homotopy equivalence. By Proposition \ref{prop join}, we have that the inclusion
$$I( (P \setminus c_{t-1,2}) \sqcup H) \hookrightarrow I(P \sqcup H) = I(X_t \setminus N[a_{t-2}])$$
is a homotopy equivalence. Since the graph $(P \setminus c_{t-1}) \sqcup H$ has no vertices adjacent to $c_{t-1,1} \in V(X_t \setminus a_{t-2})$, we have that $I((P \setminus c_{t-1,2}) \sqcup H)$ is contained in the star at $c_{t-1,1}$ in $I(X_t \setminus a_{t-2})$. This means that the composite of the inclusions
$$I(P \setminus c_{t-1,2}) \sqcup H) \xrightarrow{\simeq} I(P \sqcup H) \to I(X_t \setminus N[a_{t-2}]) \to I(X_t \setminus a_{t-2})$$
is null-homotopic. This completes the proof of (1).

Next we show (2). Since $I(P) = I(P_{12m+3}) \simeq S^{4m}$, we have
$$I(X_t \setminus N[a_{t-1}]) = I(P) * I(H) \simeq \Sigma^{4m+1} I(H).$$
In $H$, the vertices $b_{t-2, 3m}, \cdots, b_{t-2,2}$ form a $(3m-1)$-string based at $b_{t-2,2}$, and the vertices $c_{t-2, 3m}, \cdots, c_{t-2,2}$ form a $(3m-1)$-string based at $c_{t-2,2}$. By Proposition \ref{prop string}, we can delete the vertices $b_{t-2,1}, c_{t-2,1}$ from $X_t \setminus N[a_{t-2}]$, and hence we have
$$I(H) \simeq I(P_{3m-1} \sqcup P_{3m-1} \sqcup G_{t-3}) = S^{m-1} * S^{m-1} * I(G_{t-3}) = \Sigma^{2m} I(G_{t-3}).$$
This implies $I(X_t \setminus N[a_{t-2}]) \simeq \Sigma^{4m+1} I(H) = \Sigma^{6m+1} I(G_{t-3})$. This completes the proof.
\end{proof}

\begin{figure}[t]
\begin{center}
\begin{picture}(360,100)(0,-20)

\multiput(10,10)(40,0){4}{\circle*{3}}
\multiput(10,70)(40,0){4}{\circle*{3}}
\multiput(30,40)(40,0){2}{\circle*{3}}

\put(110,40){\circle{3}}
\put(150,40){\circle*{3}}

\multiput(10,10)(40,0){2}{\line(2,3){40}}
\multiput(10,70)(40,0){2}{\line(2,-3){40}}
\multiput(10,10)(0,60){2}{\line(1,0){120}}
\put(150,40){\line(-2,3){20}}
\put(150,40){\line(-2,-3){20}}

\put(10,10){\line(-2,3){10}}
\put(10,10){\line(-1,0){10}}

\put(10,70){\line(-2,-3){10}}
\put(10,70){\line(-1,0){10}}

\put(60,-8){$X_t$}


\multiput(210,10)(40,0){4}{\circle*{3}}
\multiput(210,70)(40,0){4}{\circle*{3}}
\put(230,40){\circle*{3}}
\put(270,40){\circle{3}}
\put(310,40){\circle{3}}
\put(350,40){\circle*{3}}

\put(210,10){\line(2,3){40}}
\put(210,70){\line(2,-3){40}}
\multiput(210,10)(0,60){2}{\line(1,0){120}}
\put(350,40){\line(-2,3){20}}
\put(350,40){\line(-2,-3){20}}

\put(210,10){\line(-2,3){10}}
\put(210,10){\line(-1,0){10}}

\put(210,70){\line(-2,-3){10}}
\put(210,70){\line(-1,0){10}}

\put(260,-8){$Y_t$}

\end{picture}

{\bf Figure 4.}
\end{center}
\end{figure}

\begin{lem} \label{lem 3.2.3}
For $t \ge 4$, there is a homotopy equivalence
$$I(Y_t) \simeq \Sigma^{6m+2} I(G_{t-3}) \vee \Sigma^{8m+3}I(G_{t-4}).$$
\end{lem}
\begin{proof}
By Proposition \ref{prop cofib}, it suffices to show the following assertions:
\begin{itemize}
\item[(1)] The inclusion $I(Y_t \setminus N[a_t]) \hookrightarrow I(Y_t \setminus a_t)$ is null-homotopic.

\item[(2)] $I(Y_t \setminus a_t) \simeq \Sigma^{6m+2} I(G_{t-3})$

\item[(3)] $I(Y_t \setminus N[a_t]) \simeq \Sigma^{8m+2} I(G_{t-4})$
\end{itemize}
We first show (1). In $Y_t \setminus N[a_t]$, the vertices
$$b_{t,3m}, \cdots, b_{t,1}, b_{t-1,3m+1}, \cdots, b_{t-1,1}, b_{t-2,3m+1}, \cdots b_{t-2,1},$$
form a $(9m + 2)$-string $P$ based at $b_{t-2,1}$, and the vertices
$$c_{t,3m}, \cdots, c_{t,1}, c_{t-1,3m+1}, \cdots, c_{t-1,1}, c_{t-2,3m+1}, \cdots, c_{t-2,1}$$
form a $(9m + 2)$-string $Q$ based at $c_{t-2,1}$. By Proposition \ref{prop string}, we can delete the vertices $a_{t-3}, b_{t-3, 3m+1}, c_{t-3, 3m+1}$ from $Y_t \setminus N[a_t]$. Thus we have a homotopy equivalence $I(Y_t \setminus N[a_t]) \simeq I(P) * I(Q) * I(H)$, where $H$ is the graph satisfying
$$ P \sqcup Q \sqcup H = Y_t \setminus \big( N[a_t] \cup \{ a_{t-3}, b_{t-3,3m+1}, c_{t-3,3m+1}\} \big)$$
Let $P'$ be the subgraph of $Y_t \setminus a_t$ induced by $V(P) \cup b_{t, 3m+1}$ and $Q'$ the subgraph of $Y_t \setminus a_t$ induced by $V(Q) \cup c_{t, 3m}$. Put $P'' = P' \setminus b_{t-2,1}$ and $Q'' = Q' \setminus c_{t-2,1}$. Figure 5 depicts the graphs $P \sqcup Q \sqcup H$, $P' \sqcup Q' \sqcup H$, and $P'' \sqcup Q'' \sqcup H$ when $m = 0$. 

\begin{figure}[b]
\begin{center}
\begin{picture}(400,65)(0,-15)

\multiput(0,10)(0,40){2}{\circle*{3}}
\multiput(20,10)(0,40){2}{\circle*{3}}
\multiput(40,10)(0,40){2}{\circle{3}}
\multiput(60,10)(0,40){2}{\circle*{3}}
\multiput(80,10)(0,40){2}{\circle*{3}}
\multiput(100,10)(0,40){2}{\circle{3}}

\put(20,10){\line(-1,0){27}}
\put(20,50){\line(-1,0){27}}
\put(0,10){\line(1,2){20}}
\put(0,50){\line(1,-2){20}}
\put(20,10){\line(1,2){10}}
\put(20,50){\line(1,-2){10}}

\put(30,30){\circle*{3}}
\put(10,30){\circle*{3}}
\put(50,30){\circle{3}}

\put(55,28){\tiny $a_{t-3}$}

\multiput(60,10)(0,40){2}{\line(1,0){20}}

\put(20,-10){\small $P \sqcup Q \sqcup H$}

\multiput(150,10)(0,40){2}{\circle*{3}}
\multiput(170,10)(0,40){2}{\circle*{3}}
\multiput(190,10)(0,40){2}{\circle{3}}
\multiput(210,10)(0,40){2}{\circle*{3}}
\multiput(230,10)(0,40){2}{\circle*{3}}
\multiput(250,10)(0,40){2}{\circle*{3}}

\put(170,10){\line(-1,0){27}}
\put(170,50){\line(-1,0){27}}
\put(150,10){\line(1,2){20}}
\put(150,50){\line(1,-2){20}}
\put(170,10){\line(1,2){10}}
\put(170,50){\line(1,-2){10}}

\put(180,30){\circle*{3}}
\put(160,30){\circle*{3}}
\put(200,30){\circle{3}}

\multiput(210,10)(0,40){2}{\line(1,0){40}}

\put(205,28){\tiny $a_{t-3}$}

\put(170,-10){\small $P' \sqcup Q' \sqcup H$}

\multiput(300,10)(0,40){2}{\circle*{3}}
\multiput(320,10)(0,40){2}{\circle*{3}}
\multiput(340,10)(0,40){2}{\circle{3}}
\multiput(360,10)(0,40){2}{\circle{3}}
\multiput(380,10)(0,40){2}{\circle*{3}}
\multiput(400,10)(0,40){2}{\circle*{3}}

\put(320,10){\line(-1,0){27}}
\put(320,50){\line(-1,0){27}}
\put(300,10){\line(1,2){20}}
\put(300,50){\line(1,-2){20}}
\put(320,10){\line(1,2){10}}
\put(320,50){\line(1,-2){10}}

\put(330,30){\circle*{3}}
\put(310,30){\circle*{3}}
\put(350,30){\circle{3}}

\multiput(380,10)(0,40){2}{\line(1,0){20}}

\put(355,28){\tiny $a_{t-3}$}

\put(320,-10){\small $P'' \sqcup Q'' \sqcup H$}

\end{picture}

{\bf Figure 5.}
\end{center}
\end{figure}

By Corollary \ref{cor P}, the four inclusions
$$I(P) \hookrightarrow I(P') \hookleftarrow I(P''), \; I(Q) \hookrightarrow I(Q') \hookleftarrow I(Q'')$$
are homotopy equivalences. Thus we have a commutative diagram
$$\xymatrix{
I(P \sqcup Q \sqcup H) \ar[r]^{\simeq} \ar[rd] & I(P' \sqcup Q' \sqcup H) \ar[d] & I(P'' \sqcup Q'' \sqcup H) \ar[ld] \ar[l]_{\simeq} \\
{} & I(Y_t \setminus a_t) & {} 
}$$
Here $P'' \sqcup Q'' \sqcup H$ has no vertices adjacent to $a_{t-3} \in V(Y_t \setminus a_t)$. Thus the inclusion $I(P'' \sqcup Q'' \sqcup H) \hookrightarrow I(Y_t \setminus a_t)$ is null-homotopic. By the commutativity of the above diagram, the composite of the inclusions $I(P \sqcup Q \sqcup H) \xrightarrow{\simeq} I(Y_t \setminus N[a_t]) \to I(Y_t \setminus a_t)$ is null-homotopic. This completes the proof of (1).

Next we prove (2). In $Y_t \setminus a_t$, the vertices
$$b_{t, 3m+1} , \cdots, b_{t,1}, b_{t-1,3m+1}, \cdots, b_{t-1,1}, b_{t-2,3m+1}, \cdots, b_{t-2,2}$$
form a $(9m + 2)$-string based at $b_{t-2,2}$, and the vertices
$$c_{t, 3m+1} , \cdots, c_{t,1}, c_{t-1,3m+1}, \cdots, c_{t-1,1}, c_{t-2,3m+1}, \cdots, c_{t-2,2}$$
form a $(9m+2)$-string based at $c_{t-2,2}$. By Proposition \ref{prop string}, we have homotopy equivalences
$$I(Y_t \setminus a_t) \simeq I(P_{9m+2} \sqcup P_{9m+2} \sqcup G_{t-3}) = S^{3m} * S^{3m} * I(G_{t-3}) = \Sigma^{6m+2} I(G_{t-3}).$$
This completes the proof of (2).

Finally, we prove (3). Recall $I(Y_t \setminus N[a_t]) \simeq I(P) * I(Q) * I(H)$ and $P \cong Q \cong P_{9m + 2}$. Since $I(P_{9m+2}) \simeq S^{3m}$, we have that $I(Y_t \setminus N[a_t]) \simeq \Sigma^{6m+2} I(H)$. In $H$, the vertices $b_{t-3,3m}, \cdots, b_{t-3,2}$ form a $(3m-1)$-string based at $b_{t-3,2}$ and the vertices $c_{t-3,3m}, \cdots, c_{t-3,2}$ form a $(3m-1)$-string based at $c_{t-3,2}$. By Proposition \ref{prop cofib}, we can delete the vertices $b_{t-3,1}$ and $c_{t-3,1}$ from $H$. Hence we have
$$I(H) \simeq I(P_{3m-1}) * I(P_{3m-1}) * I(G_{t-4}) \simeq S^{m-1} * S^{m-1} * I(G_{t-4}) \simeq \Sigma^{2m} I(G_{t-4}).$$
Thus we have $I(Y_t \setminus N[a_t]) \simeq \Sigma^{6m+2} I(H) \simeq \Sigma^{8m+2} I(G_{t-4})$. This completes the proof.
\end{proof}

\noindent
{\it Proof of Theorem \ref{thm 4.1.1}.} Combining the Lemma \ref{lem 3.2.1}, Lemma \ref{lem 3.2.2}, and Lemma \ref{lem 3.2.3}, we have
$$I(G_t) \simeq I(X_t) \simeq I(Y_t) \vee \Sigma^{6m+2} I(G_{t-3}) \simeq \Sigma^{6m+2} I(G_{t-3}) \vee \Sigma^{6m+2} I(G_{t-3}) \vee \Sigma^{8m+3} I(G_{t-4}).$$
This completes the proof. \qed

\vspace{2mm}
Theorem \ref{thm 4.1.1} implies that $G_t$ is recursively determined by $I(G_0)$, $I(G_1)$, $I(G_2)$, and $I(G_3)$. These complexes are easily determined by Lemma \ref{lem 3.2.2} and Lemma \ref{lem 3.2.3}.

\begin{prop} \label{prop 4.1.2}
$I(G_0) \simeq \pt$, $I(G_1) \simeq S^{2m}$, $I(G_2) \simeq S^{4m+1} \vee S^{4m+1}$, $I(G_3) \simeq S^{6m+2}$
\end{prop}
\begin{proof}
By the definition of $G_0$, it follows that $I(G_0) = \pt$. Since $G_1 \cong C_{6m+4}$, we have that $I(G_1) \simeq S^{2m}$. Since $I(G_2) \simeq I(X_2) \cong I(C_{12m+6})$, we have $I(G_2) \simeq S^{4m+1} \vee S^{4m+1}$. By Lemma \ref{lem 3.2.1} and Lemma \ref{lem 3.2.2}, we have that $I(G_3) \simeq I(X_3) \simeq I(Y_3) \vee \Sigma^{6m+2} I(G_0) \simeq I(Y_3) \cong I(C_{18m+8})$. Thus we have $I(G_3) \simeq S^{6m+2}$. This completes the proof.
\end{proof}

Finally, we determine the connectivity of $I(G_t)$. Let $k_t$ be the connectivity of $I(G_t)$. Then the following hold:

\begin{thm} \label{thm connectivity 1}
Define $s \ge 0$ and $\varepsilon = 1, 2, 3$ by $t = 3s + \varepsilon$. Then
\begin{eqnarray} \label{eqn 2}
k_t = 2mt + 2(s-1) + \varepsilon
\end{eqnarray}
\end{thm}
\begin{proof}
It follows from Proposition \ref{prop 4.1.2} that the equation (\ref{eqn 2}) holds for $s = 0$. By Theorem \ref{thm 4.1.1}, we have the equation
$$k_t = \min \{ k_{t-3} + 6m+2, k_{t-4} + 8m + 3\}$$
for $t \ge 4$. Using this, one can show that the equation (\ref{eqn 2}) holds for every $t$ by induction on $s$. This completes the proof.
\end{proof}

\begin{rem}
Jeli\'c Milutinovi\'c et al.~\cite{JMMV} showed that when $n = 3m+2$ and $t = 3s + \varepsilon$ with $\varepsilon = 1,2,3$, the complex $I(G_t)$ the connectivity is at least
$$2mt + t - \Big\lfloor \frac{t+1}{2}\Big\rfloor -1.$$
Thus the connectivity of $I(G_t)$ is strictly greater than their lower bound. On the other hand, in the case of $n = 3m$, we will see that their lower bound coincides the strict connectivity of $I(G_t)$.
\end{rem}

\subsection{Case $n = 3m$}

The purpose in this section is to determine the homotopy types and connectivities of $I(G_t)$ when $n = 3m$. Recall that the graph $H_t$ {is} the induced subgraph of $G_{t+1}$ whose vertex set if $V(G_t) \cup \{ b_{t+1, 1}, c_{t+1,1}\}$ (see Section 1). The recursive formulae of $I(G_t)$ are described as the following two theorems:

\begin{thm}\label{thm G_t}
For $t \ge 4$, there is a homotopy equivalence
$$I(G_t) \simeq \Sigma^{4m-2} I(H_{t-2}) \vee \Sigma^{6m-2} I(G_{t-3}) \vee \Sigma^{8m-3} I(H_{t-4}).$$
\end{thm}

\begin{thm}\label{thm H_t}
For $t \ge 1$, there is a homotopy equivalence
$$I(H_t) \simeq \Sigma^{2m} I(G_{t-1}) \vee \Sigma^{2m-1} I(H_{t-1}).$$
\end{thm}

We first show Theorem \ref{thm G_t}. Put $X_t = G_t \setminus a_{t-1}$.

\begin{lem} \label{lem 3m.1}
If $t \ge 2$, then $I(G_t) \simeq I(X_t) \vee \Sigma^{4m-{2}} I(H_{t-2})$.
\end{lem}
\begin{proof}
By Proposition \ref{prop cofib}, it suffices to show the following assertions:
\begin{itemize}
\item[(1)] The inclusion $I(G_t \setminus N[a_{t-1}]) \hookrightarrow I(G_t \setminus a_{t-1}) = I(X_t)$ is null-homotopic.

\item[(2)] $I(G_t \setminus N[a_{t-1}]) \simeq \Sigma^{4m-3} I(H_{t-2})$
\end{itemize}
First, we write $P$ to indicate the connected component of $G_t \setminus N[a_{t-1}]$ consisting of the vertices
$$b_{t,2}, \cdots, b_{t,3m-1}, a_t, c_{t,3m-1}, \cdots, c_{t,2}.$$
Put $P' = P \setminus c_{t,2}$ and $G_t \setminus N[a_{t-1}] = P \sqcup Q$. Since $P$ is isomorphic to $P_{6m-3}$, the inclusion $I(P') \hookrightarrow I(P)$ is a homotopy equivalence. Therefore we have that $I(P' \sqcup Q) \hookrightarrow I(P \sqcup Q) = I(G_t \setminus N[a_{t-1}])$ is a homotopy equivalence. However, $P' \sqcup Q$ has no vertices adjacent to $c_{t,1} \in V(X_t)$. Therefore, we have that the composite $I(P' \sqcup Q) \xrightarrow{\simeq} I(P \sqcup Q) \hookrightarrow I(X_t)$ is null-homotopic, and hence we have that the inclusion $I(G_t \setminus N[a_{t-1}]) = I(P \sqcup Q) \hookrightarrow I(X_t)$ is null-homotopic. This completes the proof of (1).

In $Q$, $b_{t-1, 3m-2}, \cdots, b_{t-1,3}$ is a $(3m-4)$-string based at $b_{t-1,3}$, and $c_{t-1,3m-2}, \cdots, c_{t-1,3}$ is a $(3m-4)$-string based at $c_{t-1,3}$. By Proposition \ref{prop string}, one can delete $b_{t-1,2}$ and $c_{t-1,2}$ from $Q$, and hence we have
$$I(Q) \simeq I(Q \setminus \{ b_{t-1,2}, c_{t-1,2} \}) \simeq I(P_{3m-4}) * I(P_{3m-4}) * I(H_{t-2}) = \Sigma^{2m-2} I(H_{t-2}).$$
Therefore we have
$$I(G_t \setminus N[a_{t-1}]) = I(P) * I(Q) \cong I(P_{6m-3}) * I(Q) \simeq S^{2m-2} * \Sigma^{2m-2} I(H_t) = \Sigma^{4m-3} I(H_{t-2}).$$
This completes the proof of (2).
\end{proof}

Next put $Y_t = X_t \setminus a_{t-2}$ for $t \ge 2$.

\begin{lem} \label{lem 3m.2}
If $t \ge 2$, then $I(X_t) \simeq I(Y_t)$
\end{lem}
\begin{proof}
By Proposition \ref{prop cofib}, it suffices to show that $I(X_t \setminus N[a_{t-2}])$ is contractible. In fact, $X_t \setminus N[a_{t-2}]$ has a connected component $P$ consisting of the vertices
$$b_{t-1,2}, \cdots, b_{t,3m-1}, b_{t,1}, \cdots, b_{t,3m-1} , a_t, c_{t,3m-1}, c_{t,1}, c_{t-1,3m-1}, c_{t-1,2}.$$
Then $P$ is isomorphic to $P_{12m-5}$. Since $I(P_{12m-5})$ is contractible, we have that $I(X_t \setminus N[a_{t-2}])$ is contractible.
\end{proof}

\begin{lem} \label{lem 3m.3}
If $t \ge 4$, then $I(Y_t) \simeq \Sigma^{6m-2} I(G_{t-3}) \vee \Sigma^{8m-3} I(H_{t-4})$.
\end{lem}
\begin{proof}
The proof of this lemma is almost the same as the proof of Lemma \ref{lem 3.2.3}. Put $Z_t = Y_t \setminus a_t$ and $W_t = Y_t \setminus N[a_t]$. By Proposition \ref{prop cofib}, it suffices to show the following assertions:
\begin{itemize}
\item[(1)] The inclusion $I(W_t) \hookrightarrow I(Z_t)$ is null-homotopic.

\item[(2)] $I(Z_t) \simeq \Sigma^{6m-2} I(G_{t-3})$

\item[(3)] $I(W_t) \simeq \Sigma^{8m-4} I(H_{t-4})$
\end{itemize}
In $W_t$, the vertices
$$P = b_{t, 3m-2}, \cdots, b_{t,1}, b_{t-1,3m-1}, \cdots, b_{t-1,1}, b_{t-2, 3m-1}, \cdots b_{t-2,1}$$
form a $(9m-4)$-string based at $b_{t-2,1}$, and the vertices
$$Q = c_{t, 3m-2}, \cdots, c_{t,1}, c_{t-1,3m-1}, \cdots, c_{t-1,1}, c_{t-2,3m-1}, \cdots, c_{t-2,1}$$
form a $(9m-4)$-string based at $c_{t-2,1}$. By Proposition \ref{prop cofib}, we have
$$I(W_t) \simeq I(W_t \setminus \{ a_{t-3}, b_{t-3,3m-1}, c_{t-3,3m-1}\}).$$
We put $U_t = W_t \setminus \{ a_{t-2}, b_{t-3,3m-1}, c_{t-3, 3m-1} \}$. Note that $U_t$ has two connected components
$$b_{t,3m-2}, \cdots, b_{t,1}, b_{t-1,3m-1}, \cdots, b_{t-1,1}, b_{t-2, 3m-1}, \cdots, b_{t-2, 1}$$
and
$$c_{t,3m-2}, \cdots, c_{t,1}, c_{t-1,3m-1}, \cdots, c_{t-1,1}, c_{t-2,3m-1}, \cdots, c_{t-2, 1}$$
which are isomorphic to $P_{9m-4}$. {Then define $U'_t$ to be} the induced subgraph of $Y_t$ whose vertex set is $V(U_t) \cup \{ b_{t,3m-1}, c_{t,3m-1}\}$, and $U''_t = U_t \setminus \{ b_{t-2,1}, c_{t-2,1}\}$. Since the inclusion $I(P_{9m-4}) \hookrightarrow I(P_{9m-3})$ is a homotopy equivalence, we have that the inclusions $I(U_t) \hookrightarrow I(U'_t)$ and $I(U''_t) \hookrightarrow I(U'_t)$ {are homotopy equivalences}. Thus we have the following commutative diagram:
$$\xymatrix{
I(U''_t) \ar[r]^{\simeq} \ar[rd]  & I(U'_t) \ar[d] & I(U_t) \ar[l]_{\simeq} \ar[ld] \\
{} & I(Z_t)
}$$
Here every arrow in this diagram is an inclusion. Since $U''_t$ has no vertices adjacent to $a_{t-3} \in V(Z_t)$, we have that the inclusion $I(U''_t) \hookrightarrow I(Z_t)$ is null-homotopic. Thus the above diagram {shows} that the composite of $I(U_t) \xrightarrow{\simeq} I(W_t) \hookrightarrow I(Z_t)$ is null-homotopic. This completes the proof of (1).

\begin{figure}[b]
\begin{picture}(380,120)(0,0)
\multiput(0,120)(20,0){8}{\circle*{3}}
\multiput(0,80)(20,0){8}{\circle*{3}}

\put(10,100){\circle*{3}}
\put(50,100){\circle*{3}}

\multiput(90,100)(40,0){3}{\circle{3}}

\multiput(0,80)(40,0){2}{\line(1,2){20}}
\multiput(0,120)(40,0){2}{\line(1,-2){20}}

\put(140,80){\line(-1,0){147}}
\put(140,120){\line(-1,0){147}}

\put(160,80){\circle{3}}
\put(160,120){\circle{3}}

\put(175,98){\tiny $a_t$}

\multiput(220,120)(20,0){2}{\circle*{3}}
\multiput(220,80)(20,0){2}{\circle*{3}}

\put(240,80){\line(-1,0){27}}
\put(240,120){\line(-1,0){27}}
\put(220,80){\line(1,2){20}}
\put(220,120){\line(1,-2){20}}
\put(260,80){\circle{3}}
\put(260,120){\circle{3}}

\multiput(280,80)(20,0){5}{\circle*{3}}
\multiput(280,120)(20,0){5}{\circle*{3}}

\multiput(280,80)(0,40){2}{\line(1,0){80}}
\multiput(380,80)(0,40){2}{\circle{3}}
\multiput(270,100)(40,0){4}{\circle{3}}
\put(325,125){\small $P$}
\put(325,85){\small $Q$}

\put(240,96){\small $W'_t$}

\put(-25,98){\small $W_t$}
\put(-25,28){\small $U'_t$}
\put(400,96){\small $U_t$}
\put(400,26){\small $U''_t$}

\multiput(0,50)(20,0){2}{\circle*{3}}
\multiput(0,10)(20,0){2}{\circle*{3}}
\multiput(20,10)(0,40){2}{\line(-1,0){27}}

\put(0,10){\line(1,2){20}}
\put(0,50){\line(1,-2){20}}

\multiput(60,10)(0,40){2}{\line(1,0){100}}

\multiput(50,30)(40,0){4}{\circle{3}}

\multiput(40,10)(0,40){2}{\circle{3}}

\multiput(60,10)(20,0){6}{\circle*{3}}
\multiput(60,50)(20,0){6}{\circle*{3}}

\multiput(220,50)(20,0){2}{\circle*{3}}
\multiput(220,10)(20,0){2}{\circle*{3}}

\put(240,10){\line(-1,0){27}}
\put(240,50){\line(-1,0){27}}
\put(220,10){\line(1,2){20}}
\put(220,50){\line(1,-2){20}}
\put(260,10){\circle{3}}
\put(260,50){\circle{3}}

\multiput(280,10)(0,40){2}{\circle{3}}
\multiput(300,10)(20,0){5}{\circle*{3}}
\multiput(300,50)(20,0){5}{\circle*{3}}

\multiput(300,10)(0,40){2}{\line(1,0){80}}
\multiput(270,30)(40,0){4}{\circle{3}}

\end{picture}
\begin{center}
{\bf Figure 5.}
\end{center}
\end{figure}

Next we show (3). We put $W_t = P \sqcup Q \sqcup W'_t$. Here $P$ and $Q$ are the connected components defined in the proof of (1) which are isomorphic to $P_{9m-4}$. In $W'_t$, the vertices $b_{t-3, 3m-2}, \cdots, b_{t-3,3}$ form a $(3m-4)$-string based at $b_{t-3,3}$, and the vertices $c_{t-3,3m-2}, \cdots, c_{t-3,2}$ form a $(3m-4)$-string based at $c_{t-3,3}$. Therefore we have
$$I(W'_t) \simeq I(W'_t \setminus \{ b_{t-3,2}, c_{t-3,2}\}) = I(P_{3m-4}) * I(P_{3m-4}) * I(H_{t-4}) = \Sigma^{2m-2} I(H_{t-4}).$$
Therefore we have
$$I(W_t) \simeq I(P) * I(Q) * I(W'_t) = I(P_{9m-4}) * I(P_{9m-4}) * I(W'_t) \simeq \Sigma^{8m-4} I(H_{t-4}).$$

Finally, we show (2). In $Z_t$, the vertices
$$b_{t,3m-1}, \cdots, b_{t,1}, b_{t-1,3m-1}, \cdots, b_{t-1,1}, b_{t-2,3m-1}, \cdots, b_{t,2}$$
form a $(9m-4)$-string based at $b_{t,2}$, and the vertices
$$c_{t,3m-1}, \cdots, c_{t,1}, c_{t-1,3m-1}, \cdots, c_{t-1,1}, c_{t-2,3m-1}, \cdots, c_{t,2}$$
form a $(9m-4)$-string based at $c_{t,2}$. By Proposition \ref{prop string}, we can delete the vertices $b_{t-2,1}$ and $c_{t-2,1}$. Thus we have
$$I(Z_t) \simeq I(P_{9m-4}) * I(P_{9m-4}) * I(G_{t-3}) \simeq S^{3m-2} * S^{3m-2} * I(G_{t-3}) \simeq \Sigma^{6m-2} I(G_{t-3}).$$
\end{proof}

\noindent
{\it Proof of Theorem \ref{thm G_t}.} By Lemma \ref{lem 3m.1}, Lemma \ref{lem 3m.2}, and Lemma \ref{lem 3m.3}, we have homotopy equivalences
\begin{eqnarray*}
I(G_t) & \simeq & I(X_t) \vee \Sigma^{4m-2} I(H_{t-2}) \\ 
& \simeq & I(Y_t) \vee \Sigma^{4m-2} I(H_{t-2}) \\
& \simeq & \Sigma^{6m-2} I(G_{t-3}) \vee \Sigma^{8m-3} I(H_{t-4}) \vee \Sigma^{4m-2} I(H_{t-2}).
\end{eqnarray*}
This completes the proof. \qed

\vspace{2mm}
\noindent
{\it Proof of Theorem \ref{thm H_t}.} Note that $b_{t+1,1}$ is a simplicial vertex (see Section 2 for the definition) and $N_{H_t}(b_{t+1,1}) = \{ a_t, b_{t, 3m-1}\}$. Thus by Proposition \ref{thm simplicial}, there is a homotopy equivalence
$$I(H_t) \simeq \Sigma I(H_t \setminus N[b_{t,3m-1}]) \vee \Sigma I(H_t \setminus N[a_t])$$
In $H_t \setminus N[b_{t, 3m-1}]$, the vertices $b_{t, 3m-3}, \cdots, b_{t,2}$ form a $(3m-4)$-string based at $b_{t,2}$, and the vertices $c_{t+1,1}, c_{t,3m-1}, c_{t, 3m-1}, \cdots, c_{t,2}$ {form} a $(3m-1)$-string based at $b_{t,2}$. By Proposition \ref{prop string}, we can delete the vertices $b_{t,1}$ and $c_{t,1}$, and hence we have
$$I(H_t \setminus N[b_{t, 3m-1}]) \simeq I(P_{3m-4}) * I(P_{3m-1}) * I(G_{t-1}) \simeq S^{m-2} * S^{m-1} * I(G_{t-1}) \simeq \Sigma^{2m-1} I(G_{t-1}).$$

On the other hand, in $H_t \setminus N[a_{t}]$, the vertices $b_{t,3m-2}, \cdots, b_{t,3}$ form a $(3m-4)$-string based at $b_{t,3}$, and the vertices $c_{t, 3m-2}, \cdots, c_{t,3}$ form a $(3m-4)$-string based at $c_{t,3}$. By Proposition \ref{prop string}, we can delete the vertices $b_{t,2}$ and $c_{t,2}$, and hence we have
$$I(H_t \setminus N[a_t]) \simeq I(P_{3m-4}) * I(P_{3m-4}) * I(H_{t-1}) \simeq S^{m-2} * S^{m-2} * I(H_{t-1}) \simeq \Sigma^{2m-2} I(H_{t-1}).$$
Therefore we have
$$I(H_t) \simeq \Sigma I(H_t \setminus N[b_{t,3m-1}]) \vee \Sigma I(H_t \setminus N[a_t]) \simeq \Sigma^{2m-1} I(G_{t-1}) \vee \Sigma^{2m-1} I(H_{t-1}).$$
This completes the proof. \qed

\vspace{2mm}
By Theorem \ref{thm G_t} and Theorem \ref{thm H_t}, $I(G_t)$ and $I(H_t)$ are recursively determined by $I(G_0), I(G_1), I(G_2), I(G_3)$, and $I(H_1)$.

\begin{prop} \label{prop 3.2.5}
$I(G_0) = \pt$, $I(G_1) \simeq S^{2m-1} \vee S^{2m-1}$, $I(G_2) \simeq S^{4m-2} \vee S^{4m-2}$, $I(G_3) \simeq S^{6m-2} \vee S^{6m-3}$, $I(H_0) \simeq S^0$
\end{prop}
\begin{proof}
Since $G_0 = \pt$, $H_0 = P_3$, and $G_1 = C_{6m}$, we have that $I(G_0) = \pt$, $I(G_1) \simeq S^{2m-1} \vee S^{2m-1}$, and $I(H_0) \simeq S^0$. By Lemma \ref{lem 3m.1}, we have
$$I(G_2) \simeq I(X_2) \vee \Sigma^{4m-2} I(H_0) \simeq I(C_{12m-2}) \vee S^{4m-2} \simeq S^{4m-2} \vee S^{4m-2}.$$
By Theorem \ref{thm H_t}, we have $I(H_1) \simeq \Sigma^{2m} I(G_0) \vee \Sigma^{2m-1} I(H_0) \simeq S^{2m-1}$. It follows from Lemma \ref{lem 3m.1} and Lemma \ref{lem 3m.2} that
$$I(G_3) \simeq I(X_3) \vee \Sigma^{4m-2} I(H_1) \simeq I(Y_3) \vee \Sigma^{4m-2} I(H_1) = S^{6m-2} \vee S^{6m-3}$$
since $I(Y_3) = I(C_{18m-4}) \simeq S^{6m-2}$. This completes the proof.
\end{proof}

Finally, we determine the connectivity of the complexes $I(G_t)$ and $I(H_t)$. Put $k_t = \conn(I(G_t))$ and $l_t = \conn(I(H_t))$. In \cite{JMMV}, Jeli\'c {Milutinovi\'c} et al.~showed that $k_t \ge (2m-1)t - 1$. In fact, the following {holds}:

\begin{thm} \label{thm connectivity 2}
For $t\ge 1$, the following equation holds:
\begin{eqnarray} \label{eqn 3.2.1}
k_t = l_t = (2m-1)t - 1
\end{eqnarray}
\end{thm}
\begin{proof}
By Proposition \ref{prop 3.2.5}, the equation (\ref{eqn 3.2.1}) holds for $k_1$, $k_2$, $k_3$, and since $I(H_1) = S^{2m-1}$ (see the proof of Proposition \ref{prop 3.2.5}),  the equation (\ref{eqn 3.2.1}) holds for {$l_1$}. Thus it suffices to show the following:
\begin{itemize}
\item[(a)] Suppose $t \ge 4$. If the equation (\ref{eqn 3.2.1}) holds for $t-2, t-3$, and $t-4$, then the equation (\ref{eqn 3.2.1}) holds for $k_t$.

\item[(b)] Suppose $t \ge 2$. If the equation (\ref{eqn 3.2.1}) holds for $t-1$, then the equation holds for $l_t$.
\end{itemize}

By Theorem \ref{thm G_t}, the equation
$$k_t = \min \{ l_{t-2} + 4m-2, \; k_{t-3} + 6m-2, \; l_{t-4} + 8m-3\},$$
holds for $t \ge 4$. If the equation \ref{eqn 3.2.1} holds for $t-1$, $t-2$, and $t-3$, then we have
$$k_t = \min \{ (2m-1)t-1, (2m-1)t, (2m - 1)t\} = (2m-1) t - 1.$$
Thus (a) follows. Next we prove (b). By Theorem \ref{thm H_t}, the equation
$$l_t = \min \{ k_{t-1} + 2m, l_{t-1} + 2m-1\}$$
holds for $t \ge 1$. If the equation \ref{eqn 3.2.1} holds for $t -1$, then we have
$$l_t = \min \{ (2m-1) t, (2m-1)t - 1\} = (2m-1)t - 1.$$
This completes the proof.
\end{proof}

In particular, we determine the precise connectivity of $(1 \times 1 \times t)$-honeycomb graph. Jeli\'c Milutinovi\'c et al.~\cite{JMMV} shows that the matching complex of $(1 \times 1 \times t)$-honeycomb graph is $(t-1)$-connected (see Corollary 4.4 in \cite{JMMV}). Thus we show that $(1 \times 1 \times t)$-honeycomb graph is not $t$-connected:

\begin{cor} \label{cor conn honeycomb}
The connectivity of $(1 \times 1 \times t)$-honeycomb graph is $t-1$.
\end{cor}

\end{document}